\documentclass[11pt]{amsart}
\usepackage[margin=1in]{geometry}

\usepackage{amssymb}
\usepackage{amsthm}
\usepackage{amsmath}
\usepackage{mathrsfs}
\usepackage{amsbsy}
\usepackage{bm}
\usepackage{hyperref}
\usepackage{tikz}
\usepackage{array}
\usepackage{enumerate}
\usepackage{bbm}
\usepackage{comment}
\usepackage{mathtools}
\usepackage{makecell} 
\usepackage{colortbl}
\usepackage{xcolor}
\usepackage{tabularray}

\DeclareFontFamily{U}{mathx}{}
\DeclareFontShape{U}{mathx}{m}{n}{<-> mathx10}{}
\DeclareSymbolFont{mathx}{U}{mathx}{m}{n}
\DeclareMathAccent{\widecheck}{0}{mathx}{"71}

\definecolor{lavender}{rgb}{0.4,0,1}
\definecolor{MyGreen}{rgb}{0,0.75,0}

\hypersetup{colorlinks=true, citecolor=lavender, linkcolor=lavender,urlcolor=lavender}

\usepackage{hhline}
\allowdisplaybreaks
\usepackage[noadjust]{cite}

\usepackage{caption}
\usepackage[noabbrev,capitalise,nameinlink]{cleveref}
\crefname{conjecture}{Conjecture}{Conjectures}

\newtheorem{theorem}{Theorem}[section]

\newtheorem{corollary}[theorem]{Corollary}

\newtheorem{lemma}[theorem]{Lemma}

\theoremstyle{definition}
\newtheorem{definition}[theorem]{Definition}

\newcommand{\Mac}{\mathrm{Mac}}
\newcommand{\MacPop}{\mathrm{MacPop}}
\newcommand{\Pop}{\mathrm{Pop}}
\newcommand{\U}{\mathrm{U}}
\renewcommand{\L}{\mathrm{D}}
\newcommand{\Des}{\mathrm{Des}} 
\newcommand{\leqb}{\leq_\mathsf{B}}
\newcommand{\leqr}{\leq_\mathsf{L}}
\newcommand{\G}{\mathbf{G}} 
\newcommand{\ASM}{\mathrm{ASM}}
\newcommand{\rk}{\mathrm{rk}} 
\newcommand{\Mat}{\mathrm{Mat}} 
\newcommand{\RR}{\mathcal{R}} 
 
\newcommand{\del}{\mathrm{del}} 
\newcommand{\lk}{\mathrm{lk}} 
\newcommand{\Q}{\mathsf{Q}}

\newcommand{\dfn}[1]{\textcolor{blue}{\emph{#1}}}

\title{Weak Order on the MacNeille Completion of Bruhat Order} 

\author{Colin Defant}

\begin{document}

\begin{abstract} 
Let $\mathrm{Mac}(W)$ be the MacNeille completion of the Bruhat order of a Coxeter group $W$. We introduce an action of the $0$-Hecke monoid of type $W$ on $\mathrm{Mac}(W)$, which allows us to define a weak order and a descent set statistic on $\mathrm{Mac}(W)$. When $W$ is of type~$A$, we recover constructions of Hamaker and Reiner, which were originally formulated in terms of monotone triangles and alternating sign matrices. Using this action, we prove that certain unions of Knutson--Miller subword complexes are vertex-decomposable. By specializing to type~$A$, we prove a conjecture of Escobar, Klein, and Weigandt regarding Cohen--Macaulay ASM varieties. Along the way, we also exhibit a counterexample to a conjecture of Hamaker and Reiner regarding the poset topology of intervals in the ASM weak order. When $W$ is finite, our action also gives rise to a Boolean cell complex with one facet for each element of $\mathrm{Mac}(W)$. We prove that this is an abstract simplicial complex and that every linear extension of the weak order on $\mathrm{Mac}(W)$ is a shelling order for its facets. Finally, when $W$ is finite and irreducible, we use our $0$-Hecke action to introduce a noninvertible dynamical system on $\mathrm{Mac}(W)$ that we call the \emph{MacNeille pop-stack operator}, and we prove that the maximum number of iterations of this operator needed to reach the bottom state is $h-1$, where $h$ is the Coxeter number of $W$. 

This article is meant to serve as a case study in using large language models to automate the workflow of mathematical research. The proof of the conjecture of Escobar--Klein--Weigandt and the disproof of the conjecture of Hamaker--Reiner were obtained autonomously by ChatGPT~5.4 Pro. Other aspects of the paper were obtained mostly by the author, but ChatGPT expedited the process. We provide a detailed account of this interaction.      
\end{abstract} 

\maketitle

\section{Introduction} 
\subsection{MacNeille Completions} 
Let $P$ be a poset. An \dfn{order filter} of $P$ is a subset $F\subseteq P$ such that if $x\leq y$ and $x\in F$, then $y\in F$. For $X\subseteq P$, let us write \[\U(X)=\{p\in P:p\geq x \text{ for all }x\in X\}\quad\text{and}\quad\L(X)=\{p\in P:p\leq x \text{ for all }x\in X\}.\] The \dfn{MacNeille completion} of $P$, denoted $\Mac(P)$, is the set of all order filters $F$ of $P$ such that $F=\U(\L(F))$. We view $\Mac(P)$ as a poset under the reverse containment order. That is, for $F,F'\in\Mac(P)$, we write $F\leq F'$ if $F\supseteq F'$. There is a natural poset embedding $\iota\colon P\to\Mac(P)$ given by $\iota(x)=\U(\{x\})$.  

A \dfn{complete lattice} is a poset in which any subset has a greatest lower bound and a least upper bound. It is well known that $\Mac(P)$ is the smallest complete lattice (up to isomorphism) containing a subposet isomorphic to $P$. 

\subsection{Coxeter Systems and $0$-Hecke Monoids} 
Let $(W,S)$ be a Coxeter system, and let $e$ be the identity element of $W$. 
%We have the Coxeter presentation 
%\[W=\langle S: (ss')^{m(s,s')}=e\text{ for all }s,s'\in S\rangle,\] where $m(s,s)=1$ for all $s\in s$ and $m(s,s')=m(s',s)\in\{2,3,\ldots\}\cup\{\infty\}$ whenever $s\neq s'$. 
For $w\in W$, let $\ell(w)$ denote the Coxeter length of $w$, and let \[\Des(w)=\{s\in S:\ell(sw)<\ell(w)\}\] denote the set of (left) descents of $w$. 

%The \dfn{$0$-Hecke monoid of type $W$} is the monoid $H_W(0)$ generated by elements $T_s$ for $s\in S$, subject to the following relations: 
%\begin{alignat*}{2}
%\underbrace{T_s T_{s'} \cdots}_{m(s,s')} &= \underbrace{T_{s'} T_s \cdots}_{m(s,s')} \quad&& \text{for all distinct } s, s' \in S; \\
%T_s^2 &= T_s && \text{for all } s \in S. 
%\end{alignat*}
%Given a reduced word $s_1\cdots s_k$ of an element $w\in W$, we define $T_w=T_{s_1}\cdots T_{s_k}$; the defining relations of $H_W(0)$ ensure that this is well defined. In fact, the map $w\mapsto T_W$ is a bijection from $W$ to $H_W(0)$.  

For each $s\in S$, define $\pi_s\colon W\to W$ by 
\begin{equation}\label{eq:pis}
\pi_s(w)=\begin{cases}
    sw & \text{if } s\in \Des(w) \\
    w & \text{if } s\not\in\Des(w).
\end{cases}
\end{equation} 
The monoid $H_W(0)$ generated by $\{\pi_s:s\in S\}$ is the \dfn{$0$-Hecke monoid of type $W$}. 
For each $w\in W$, we can define $\pi_w=\pi_{s_1}\cdots\pi_{s_k}$, where $s_1\cdots s_k$ is a reduced word for $w$; this is well defined. The map $w\mapsto\pi_w$ is a bijection from $W$ to $H_W(0)$. 

We view $W$ as a poset under the (strong) Bruhat order, which we denote by $\leqb$. 
For $X\subseteq W$ and $u\in W$, let 
\begin{equation}\label{eq:piuX}
\pi_u(X)=\{\pi_u(x):x\in X\},
\end{equation} 
and let 
\begin{equation}\label{eq:nabla} 
\nabla(X)=\{w\in W:x\leqb w\text{ for some }x\in X\}.
\end{equation} 

Our first theorem is the following. 
\begin{theorem}\label{thm:descents} 
Let $F\in\Mac(W)$ and $u\in W$. We have $\nabla(\pi_u(F))\in\Mac(W)$. Moreover, for every set $Y\subseteq W$, we have $\nabla(\pi_u(\nabla(Y)))=\nabla(\pi_u(Y))$.  
\end{theorem}

The importance of \cref{thm:descents} is that it allows us to make the following definition.\footnote{The author thanks Zach Hamaker for originally suggesting this definition, which he formulated together with Vic Reiner.} 

\begin{definition}\label{def:main} 
Let $W$ be a Coxeter group equipped with the Bruhat order. For each $u\in W$, define $\tau_u\colon\Mac(W)\to\Mac(W)$ by 
\[\tau_u(F)=\nabla(\pi_u(F)).\] Let the \dfn{(left) weak order} on $\Mac(W)$ be the partial order $\leqr$ defined so that $F\leqr F'$ if there exists $u\in W$ such that $F=\tau_{u}(F')$. Define the \dfn{descent set} of an order filter $F\in\Mac(W)$ to be the set 
\[\Des(F)=\{s\in S:\tau_s(F)\neq F\}.\] 
\end{definition} 

It follows from \cref{thm:descents} that the set $\{\tau_u:u\in W\}$ is a monoid isomorphic to the $0$-Hecke monoid of type~$W$. Thus, the operators $\tau_u$ define an action of the $0$-Hecke monoid of $W$ on $\Mac(W)$. 

We will prove that the poset embedding $\iota\colon W\to\Mac(W)$ satisfies $\iota(\pi_u(w))=\tau_u(\iota(w))$ for all $u,w\in W$. This ensures that the weak order and descent set statistic on $\Mac(W)$ introduced in \cref{def:main} extend the usual weak order and descent set statistics on $W$. 

\subsection{The ASM Weak Order}\label{subsec:1.3} 

An $n\times n$ \dfn{alternating sign matrix (ASM)} is an $n\times n$ matrix with entries in $\{-1,0,1\}$ such that the nonzero entries within each row or column alternate in sign and sum to $1$. Let $\ASM_n$ be the set of $n\times n$ alternating sign matrices. For $A=(A_{a,b})_{a,b\in[n]}\in\ASM_n$ and $i,j\in[n]$, define 
\[\rk_A(i,j)=\sum_{a\leq i}\sum_{b\leq j}A_{a,b}.\]
We write $A\leqb A'$ if and only if $\rk_A(i,j)\geq\rk_{A'}(i,j)$ for all $i,j\in[n]$. The partial order $\leqb$ is the \dfn{Bruhat order} on $\ASM_n$. 

We can view the symmetric group $S_n$ as a subset of $\ASM_n$ by identifying each permutation with its permutation matrix. The restriction of the Bruhat order on $\ASM_n$ to $S_n$ is the usual Bruhat order on $S_n$. It is well known that the map $\Psi\colon\ASM_n\to\Mac(S_n)$ given by \[\Psi(A)=\{w\in S_n:A\leqb w\}\] is a poset isomorphism.  

There is a classical bijection between $\ASM_n$ and a set $\mathrm{MT}_n$ of combinatorial objects known as \emph{monotone triangles of order $n$}. 
Hamaker and Reiner \cite{HR} defined an action of the $0$-Hecke monoid of $S_n$ on $\mathrm{MT}_n$, which then allowed them to define a weak order and a descent set statistic on $\mathrm{MT}_n$. By translating from monotone triangles to alternating sign matrices, this produces a $0$-Hecke action, a weak order, and a descent set statistic on $\ASM_n$. It is straightforward to verify that their definitions agree with those in \cref{def:main} when $W=S_n$.\footnote{Technically, one must take inverses because Hamaker and Reiner use the right weak order while we use the left weak order.} The weak order that Hamaker and Reiner defined on $\ASM_n$, which we also denote by $\leqr$, is now known as the \dfn{ASM weak order}. 

Given a set $J$ of simple reflections of $S_n$, we let $w_\circ(J)$ denote the long element of the parabolic subgroup generated by $J$. For $A,A'\in\ASM_n$ with $A'\leqr A$, let $[A',A]_{\mathsf{L}}$ be the interval between $A'$ and $A$ in the ASM weak order. Let ${\Delta_{\mathsf{L}}(A',A)}$ denote the order complex of the open interval ${[A',A]_{\mathsf L}\setminus\{A',A\}}$. Hamaker and Reiner conjectured that $\Delta_{\mathsf{L}}(A',A)$ is contractible unless ${A'=\tau_{w_\circ(J)}(A)}$ for some $J\subseteq\Des(A)$, in which case the order complex is homotopy equivalent to a $(|J|-2)$-dimensional sphere. In \cref{fig:counterexample}, we exhibit a counterexample to this conjecture, which we discovered using ChatGPT~5.4 Pro. The figure shows an interval $[A',A]_{\mathsf{L}}$ in the ASM weak order on $\ASM_6$ such that $\Delta_{\mathsf{L}}(A',A)$ is not contractible even though $A'$ is not of the form $\tau_{w_\circ(J)}(A)$. To see that $\Delta_{\mathsf{L}}(A',A)$ is not contractible, note that $\mu_{\mathsf{L}}(A',A)=1\neq 0$, where $\mu_{\mathsf{L}}$ is the M\"obius function of the ASM weak order. 

\begin{figure}[ht]
\begin{center}\includegraphics[width=\linewidth]{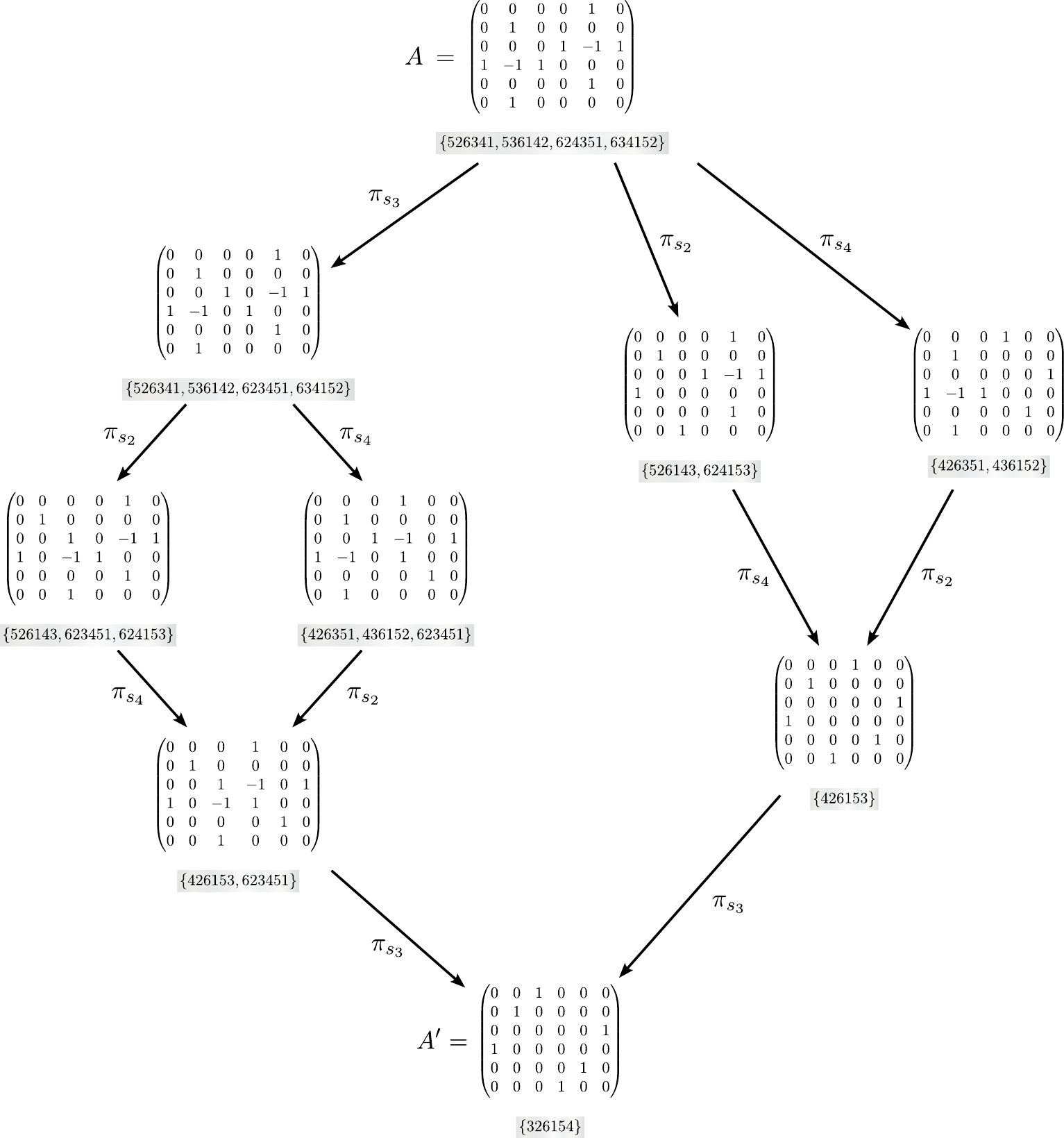}
  \end{center}
\caption{An interval $[A',A]_{\mathsf{L}}$ in the ASM weak order on $\ASM_6$ such that the order complex $\Delta_{\mathsf{L}}(A',A)$ is not contractible. Below each alternating sign matrix $B$ is the set $\min(\Psi(B))\subseteq S_6$.    }\label{fig:counterexample}
\end{figure}

\subsection{ASM Varieties} 
Fix a field $\mathbb K$, and let $\Mat_n$ be the set of $n\times n$ matrices with entries in $\mathbb K$. Associated to each alternating sign matrix $A\in\ASM_n$ is the \dfn{ASM variety} 
\[X_A=\{M\in\Mat_n:\mathrm{rank}(M_{[i],[j]})\leq\rk_A(i,j)\text{ for all }i,j\in[n]\},\] where $M_{[i],[j]}$ denotes the restriction of $M$ to its first $i$ rows and its first $j$ columns. 

Let $Z=(z_{i,j})_{i,j\in[n]}$ be an $n\times n$ matrix of variables. Let $\RR_n=\mathbb K[z_{i,j}:i,j\in[n]]$ be the polynomial ring in these variables. Let $I_k(Z_{[i],[j]})$ be the ideal of $\RR_n$ generated by the $k$-minors of the restricted matrix $Z_{[i],[j]}$. For $A\in\ASM_n$, we define the \dfn{ASM ideal} 
\[I_A=\sum_{i,j\in[n]}I_{\rk_{A}(i,j)+1}(Z_{[i],[j]}).\] Then $X_A$ is the vanishing locus of $I_A$. 

When $A\in S_n$, the ASM variety $X_A$ is called a \dfn{matrix Schubert variety}, and $I_A$ is called a \dfn{Schubert determinantal ideal}. Fulton introduced matrix Schubert varieties \cite{Fulton} and proved that they are irreducible and Cohen--Macaulay. He also proved that for $A\in S_n$, the codimension of $X_A$ is $\ell(A)$. 
One of the primary motivations for studying ASM varieties comes from a result of Weigandt, which states that they are exactly the intersections of matrix Schubert varieties \cite{Wei17}.  

Cohen--Macaulayness is a central homological regularity condition in commutative algebra and algebraic geometry.  It is weak enough to hold for many singular spaces, but strong enough to rule out several pathologies: Cohen--Macaulay varieties with standard graded coordinate rings are equidimensional and have well-behaved local cohomology and duality theory. One of our main results is as follows.  

\begin{theorem}\label{thm:CM}
Let $A\in\ASM_n$. The following are equivalent: 
\begin{enumerate}[(i)]
\item\label{(i)} The ASM weak order interval $[e,A]_{\mathsf{L}}$ is graded. 
\item\label{(ii')} For every $B\in[e, A]_{\mathsf{L}}$, the permutations in $\min(\Psi(B))$ all have the same Coxeter length. 
\item\label{(ii)} For every $B\in[e, A]_{\mathsf{L}}$, the ASM variety $X_B$ is equidimensional. 
\item\label{(iii)} For every $B\in[e, A]_{\mathsf{L}}$, the ASM variety $X_B$ is Cohen--Macaulay.  
\end{enumerate} 
\end{theorem} 

Escobar, Klein, and Weigandt proved the equivalence of \eqref{(i)}, \eqref{(ii')}, and \eqref{(ii)} in \cref{thm:CM}, and they conjectured the equivalence of these conditions to \eqref{(iii)} \cite[Conjecture~3.21]{EKW}. In general, it is still open to fully characterize Cohen--Macaulay ASM varieties. Axelrod--Freed, Hao, Kendall, Klein, and Luo made some additional progress toward this goal \cite{AFHKKL}. 

Our proof of \cref{thm:CM} was obtained autonomously by ChatGPT~5.4 Pro. 

\subsection{Subword Complexes}

Let $(W,S)$ be a Coxeter system. Given a finite word $\Q$ over the alphabet $S$ and an element $w\in W$, we consider the subword complex $\Delta(\Q,w)$, the abstract simplicial complex whose faces are the complements of the index sets of the subwords of $\Q$ that contain reduced words for $w$. Knutson and Miller introduced subword complexes \cite{KM04,KM05}, which have been studied extensively ever since \cite{BC17,EM,PS15,STW25}. For $F\in\Mac(W)$, we consider the simplicial complex 
\[\Delta(\Q,F)=\bigcup_{w\in\min(F)}\Delta(\Q,w),\] where $\min(F)$ is the set of minimal elements of $F$ (in Bruhat order).   

One of the main intermediary steps needed for the proof of \cref{thm:CM} is as follows. 

\begin{theorem}\label{thm:vertex}
Let $F\in\Mac(W)$. Assume that for every $F'\in\Mac(W)$ with $F'\leqr F$, the minimal elements of $F'$ all have the same Coxeter length. Then $\Delta(\Q,F)$ is vertex-decomposable for every finite word $\Q$ over $S$. 
\end{theorem} 

\subsection{MacNeille Pop-Stack Operators} 

Let us now assume that the Coxeter system $(W,S)$ is finite and irreducible. Let $w_\circ$ denote the long element of $W$. For $J\subseteq S$, let $W_J$ be the parabolic subgroup generated by $J$, and let $w_\circ(J)$ be the long element of $W_J$. 

\begin{definition}\label{def:pop} 
Let $(W,S)$ be a finite irreducible Coxeter system. Define the \dfn{MacNeille pop-stack operator} $\MacPop\colon \Mac(W)\to\Mac(W)$ by 
\[\MacPop(F)=\tau_{w_\circ(\Des(F))}(F).\] 
\end{definition}  

In a previous article \cite{DefantCoxeterPop}, the author introduced the \dfn{pop-stack operator} of $W$, which is the map $\Pop\colon W\to W$ defined by $\Pop(x)=\pi_{w_\circ(\Des(x))}(x)$. When $W=S_n$, this is the classical pop-stack sorting map from enumerative combinatorics and computer science \cite{CG,CGP,Elder}. We will show that \begin{equation}\label{eq:embedding2}
\iota(\Pop(x))=\MacPop(\iota(x))
\end{equation} for every $x\in W$, so the MacNeille pop-stack operator extends the pop-stack operator.\footnote{The author has introduced a pop-stack operator on any finite lattice \cite{DefantMeeting}. This definition agrees with the Coxeter-theoretic definition from \cite{DefantCoxeterPop} when the lattice is the weak order on $W$. Since $\Mac(W)$ is also a lattice, we can consider the lattice-theoretic pop-stack operator on $\Mac(W)$. This operator, however, is generally \emph{not} the same as the MacNeille pop-stack operator.}

We view $\MacPop\colon \Mac(W)\to\Mac(W)$ as a noninvertible dynamical system. If we start at an arbitrary element $F\in\Mac(W)$ and then repeatedly apply $\MacPop$, we will eventually reach the order filter $W=\iota(e)$, which is the minimum element of $\Mac(W)$. Thus, it is natural to define 
\[g_{\MacPop}(F)=\min\{t\in\mathbb Z_{\geq 0}:\MacPop^t(F)=W\}.\] We are interested in the quantity 
\[\G_{\Mac(W)}=\max_{F\in\Mac(W)}g_{\MacPop}(F).\] Similarly, for $x\in W$, we can define 
\[g_{\Pop}(x)=\min\{t\in\mathbb Z_{\geq 0}:\Pop^t(x)=e\}\] and consider the quantity 
\[\G_W=\max_{x\in W}g_{\Pop}(x).\] 
The author proved that $\G_W=h-1$, where $h$ is the \emph{Coxeter number} of $W$  \cite[Theorem~1.3]{DefantCoxeterPop}. 
It follows from \eqref{eq:embedding2} that $g_{\Pop}(x)=g_{\MacPop}(\iota(x))$ for every $x\in W$. Thus, 
\begin{equation}\label{eq:h}
h-1=\G_W\leq \G_{\Mac(W)}.
\end{equation}  

The following result shows that the maximum orbit lengths of the MacNeille pop-stack operator can be obtained by starting at an element of $\iota(W)$. 

\begin{theorem}\label{thm:pop}
Let $(W,S)$ be a finite irreducible Coxeter system with Coxeter number $h$. We have 
\[\G_{\Mac(W)}=h-1.\]
\end{theorem}

\subsection{MacNeille Complexes}
Assume that $(W,S)$ is finite. For $J\subseteq S$, let $w_\circ(J)$ be the long element of the parabolic subgroup generated by $J$, and define
\[p_J(F)=\tau_{w_\circ(J)}(F)\]
for $F\in\Mac(W)$. Following a suggestion of Hamaker and Reiner, one can use these parabolic projections to glue together simplices indexed by the elements of $\Mac(W)$. The resulting Boolean cell complex, which we call the \dfn{MacNeille complex} of $(W,S)$, has facets naturally indexed by $\Mac(W)$. When $W=S_n$, it agrees with the complex that Hamaker and Reiner constructed from partial monotone triangles. A Boolean cell complex need not be an abstract simplicial complex because distinct faces can have the same set of vertices. Our next result shows that this pathology never occurs for MacNeille complexes.

\begin{theorem}\label{thm:complex}
Let $(W,S)$ be a finite Coxeter system. The MacNeille complex of $(W,S)$ is a pure $(|S|-1)$-dimensional abstract simplicial complex. Moreover, every linear extension of the weak order on $\Mac(W)$ is a shelling order for its facets.
\end{theorem}

\subsection{Outline} In \cref{sec:defining}, we prove \cref{thm:descents}, showing that the notions in \cref{def:main} are well defined. In \cref{sec:subword}, we discuss subword complexes and prove \cref{thm:vertex}. \cref{sec:CM} discusses Stanley--Reisner rings and Cohen--Macaulayness and proves \cref{thm:CM}. We prove \cref{thm:pop} in \cref{sec:pop}. In \cref{sec:complex}, we define the MacNeille complex and prove \cref{thm:complex}. 

In \cref{sec:AI}, we discuss the use of ChatGPT in producing this manuscript, especially for proving \cref{thm:CM}. We provide a table summarizing several conversations the author had with ChatGPT in an attempt to see which types of prompt led to success. 

\section{Defining Weak Order}\label{sec:defining} 
The purpose of this section is to prove \cref{thm:descents}. We begin with some lemmas that pertain to an arbitrary poset $P$. 

\begin{lemma}\label{lem:easy} 
Let $P$ be a poset. For each subset $X\subseteq P$, we have $X\subseteq \U(\L(X))$ and $X\subseteq \L(\U(X))$. 
\end{lemma}
\begin{proof}
We will only prove that $X\subseteq \U(\L(X))$ since the proof of the other containment is completely analogous. Choose $x\in X$. We wish to show that $x\in\U(\L(X))$, so we need to show that $x\geq x'$ for every $x'\in\L(X)$. This is immediate from the definition of $\L(X)$. 
\end{proof}

\begin{lemma}\label{lem:easy2} 
Let $P$ be a poset. If $X\subseteq P$, then $\U(X)=\U(\L(\U(X)))$. 
\end{lemma}

\begin{proof}
Let $Y=\U(X)$. We know by \cref{lem:easy} that $Y\subseteq \U(\L(Y))$, so $\U(X)\subseteq \U(\L(\U(X)))$. On the other hand, because $X\subseteq \L(\U(X))$ by \cref{lem:easy}, we must have $\U(X)\supseteq \U(\L(\U(X)))$. 
\end{proof}

\begin{corollary}\label{Cor1}
An order filter $F$ of $P$ is in $\Mac(P)$ if and only if $F=\U(X)$ for some $X\subseteq P$. 
\end{corollary}

\begin{proof}
If $F\in\Mac(P)$, then $F=\U(\L(F))$ by the definition of $\Mac(P)$. Conversely, if $F=\U(X)$ for some $X\subseteq P$, then it follows from \cref{lem:easy2} that $F=\U(\L(F))$, so $F\in\Mac(P)$. 
\end{proof} 

We now specialize our attention to the case where $P$ is a Coxeter group $W$ endowed with the Bruhat order. Let $S$ be the set of simple reflections of $W$. Recall the definitions from \eqref{eq:pis}, \eqref{eq:piuX}, and \eqref{eq:nabla}. We say a map $\varphi\colon W\to W$ is \dfn{order-preserving} if for all $x,y\in W$ with $x\leqb y$, we have $\varphi(x)\leqb\varphi(y)$. For each $s\in S$, it follows from the Lifting Property of the Bruhat order \cite[Proposition~2.2.7]{BB} that the map $\pi_s\colon W\to W$ is order-preserving. As a consequence, for every $u\in W$, the map $\pi_u\colon W\to W$ is order-preserving.

\begin{lemma}\label{Lem2}
For $X\subseteq W$ and $s\in S$, we have $\U(\pi_s(X))=\nabla(\pi_s(\U(X)))$. 
\end{lemma} 

\begin{proof}
We first show that $\U(\pi_s(X))\subseteq\nabla(\pi_s(\U(X)))$. Choose $u\in \U(\pi_s(X))$. Our goal is to show that there exists $w\in \U(X)$ such that $\pi_s(w)\leqb u$. 

First, suppose $s\in \Des(u)$. If $x\in X$, then $\pi_s(x)\leqb u$ by the choice of $u$. Since $s\in \Des(u)$, we must actually have $x \leqb u$. This shows that $u\in \U(X)$, so we can put $w=u$. 

Next, suppose $s\not\in \Des(u)$. For every $x\in X$, we have $\pi_s(x) \leqb u$, so $x \leqb su$. This shows that $su\in U(X)$, so we can put $w=su$. 

We now need to prove the reverse containment $\nabla(\pi_s(\U(X)))\subseteq \U(\pi_s(X))$. The map $\pi_s$ is order-preserving. It follows that if $v\in \U(X)$, then $\pi_s(x) \leqb \pi_s(v)$ for all $x\in X$, so ${\pi_s(v)\in \U(\pi_s(X))}$. This shows that $\pi_s(\U(X))\subseteq \U(\pi_s(X))$. Because $\U(\pi_s(X))$ is an order filter, we must have $\nabla(\pi_s(\U(X)))\subseteq \U(\pi_s(X))$. 
\end{proof} 

We are now in a position to prove \cref{thm:descents}, which states that for  $F\in\Mac(W)$, $u\in W$, and $Y\subseteq W$, we have $\nabla(\pi_u(F))\in\Mac(W)$ and $\nabla(\pi_u(\nabla(Y)))=\nabla(\pi_u(Y))$.  
\begin{proof}[Proof of \cref{thm:descents}]
Choose $F\in\Mac(W)$, $u\in W$, and $Y\subseteq W$.
We proceed by induction on $\ell(u)$, noting that the case where $\ell(u)=0$ is trivial since $\pi_e$ is the identity operator. Let us first assume that $u=s\in S$. Let $X=\L(F)$. We have $F=\U(X)$, so \cref{Lem2} implies that $\nabla(\pi_s(F))=\U(\pi_s(X))$. It follows from \cref{Cor1} that $\nabla(\pi_s(F))\in\Mac(W)$. The containment $\nabla(\pi_s(\nabla(Y)))\supseteq\nabla(\pi_s(Y))$ is immediate, so we must prove the reverse containment. To do so, it suffices to show that $\pi_s(\nabla(Y))\subseteq\nabla(\pi_s(Y))$. Choose $x\in \pi_s(\nabla(Y))$, and write $x=\pi_s(z)$ for some $z\in\nabla(Y)$. There exists $y\in Y$ such that $y\leqb z$. Since $\pi_s$ is order-preserving, we have $\pi_s(y)\leqb\pi_s(z)=x$. Since $\pi_s(y)\in\pi_s(Y)$, this shows that $x\in\nabla(\pi_s(Y))$, as desired. 

Now suppose $\ell(u)>1$, and write $u=sv$ for some $s\in \Des(u)$. We know by induction that $\nabla(\pi_v(F))\in\Mac(W)$, so it follows from the preceding paragraph that 
\begin{align*}
\nabla(\pi_u(F))&=\nabla(\pi_s(\pi_v(F))) \\ 
&=\nabla(\pi_s(\nabla(\pi_v(F))))\in\Mac(W).
\end{align*} 
By repeated application of the induction hypothesis, we find that 
\begin{align*}
\nabla(\pi_u(\nabla(Y)))&=\nabla(\pi_s(\pi_v(\nabla(Y)))) \\ &=\nabla(\pi_s(\nabla(\pi_v(\nabla(Y))))) \\ 
&=\nabla(\pi_s(\nabla(\pi_v(Y)))) \\ 
&=\nabla(\pi_s(\pi_v(Y))) \\ 
&=\nabla(\pi_u(Y)), 
\end{align*} 
as desired. 
\end{proof} 

Recall that $\iota\colon W\to\Mac(W)$ is the poset embedding defined by $\iota(x)=\U(\{x\})$. For all $u,w\in W$, the fact that $\pi_u\colon W\to W$ is order-preserving implies that 
\begin{equation}\label{eq:iota}
\iota(\pi_u(w))=\tau_u(\iota(w)).
\end{equation} 

\section{Subword Complexes}\label{sec:subword} 

Let $V$ be a finite set. Let $\Delta$ be an abstract simplicial complex with vertex set $V$. That is, $\Delta$ is a collection of subsets of $V$ such that if $D\in\Delta$ and $D'\subseteq D$, then $D'\in\Delta$. The sets in $\Delta$ are called \dfn{faces}, and the inclusion-maximal faces are called \dfn{facets}. For $v\in V$, the \dfn{deletion} of $v$ from $\Delta$ is the abstract simplicial complex \[\del_v(\Delta)=\{D\in\Delta:v\not\in D\},\] and the \dfn{link} of $v$ in $\Delta$ is the abstract simplicial complex 
\[\lk_v(\Delta)=\{D\in\del_v(\Delta):(D\cup\{v\})\in\Delta\}.\] 
We say $\Delta$ is \dfn{vertex-decomposable} if either of the following hold:
\begin{itemize}
\item $\Delta$ is a simplex (i.e., has a unique facet) or is empty;
\item there exists a vertex $v$ such that $\del_v(\Delta)$ and $\lk_v(\Delta)$ are both vertex-decomposable and every facet of $\del_v(\Delta)$ is a facet of $\Delta$. 
\end{itemize} 

A \dfn{reduced word} for an element $w\in W$ is a word of length $\ell(w)$ over $S$ whose product is $w$. Let $\Q$ be a finite word over $S$. Formally, we view $\Q$ as a function from a set $V\subset\mathbb Z$ to $S$. If $V=\{k_1<\cdots<k_t\}$, then we can write $\Q=(s_{k_1},\ldots,s_{k_t})$, where $s_{k_i}$ is the image of $k_i$ under this function. For $V'\subseteq V$, we write $\Q\vert_{V'}$ for the restriction of $\Q$ to $V'$. The \dfn{subword complex} $\Delta(\Q,w)$ is the abstract simplicial complex with vertex set $V$ whose faces are the sets $D$ such that $\Q\vert_{V\setminus D}$ contains a reduced word for $w$. For $F\in\Mac(W)$, we are interested in the abstract simplicial complex 
\[\Delta(\Q,F):=\bigcup_{w\in \min(F)}\Delta(\Q,w).\] 

Our aim in this section is to prove \cref{thm:vertex}. The key point is the following lemma, which makes use of the $0$-Hecke action on $\Mac(W)$. 

\begin{lemma}\label{lem:deletion} 
Fix $F\in\Mac(W)$ and a nonempty finite word $\Q$ over $S$ with index set $V$. Write $V=\{k_1<\cdots<k_t\}$, and let $V'=V\setminus\{k_1\}$. Write $\Q=s\Q'$, where $s\in S$ and $\Q'=\Q\vert_{V'}$. We have 
\[\del_{k_1}(\Delta(\Q,F))=\Delta(\Q',\tau_s(F)).\]
\end{lemma} 
\begin{proof}
Observe that a set $D\subseteq V'$ is a face of $\del_{k_1}(\Delta(\Q,F))$ if and only if $s\Q'\vert_{V'\setminus D}$ contains a reduced word for an element of $\min(F)$. In addition, note that because $\pi_s$ is order-preserving, we have \[\min(\tau_s(F))\subseteq\{\pi_s(w):w\in\min(F)\}.\]

Suppose $D$ is a face of $\del_{k_1}(\Delta(\Q,F))$. Then $s\Q'\vert_{V'\setminus D}$ contains a reduced word $\mathsf{w}$ for an element $w$ of $\min(F)$. If $\mathsf{w}$ uses the initial $s$ in $s\Q'\vert_{V'\setminus D}$, then $s\in\Des(w)$, so $sw=\pi_s(w)\in\tau_s(F)$. Deleting the initial $s$ from $\mathsf{w}$ yields a reduced word $\mathsf{v}$ for $\pi_s(w)$ contained in $\Q'\vert_{V'\setminus D}$. There exists $y\in\min(\tau_s(F))$ such that $y\leqb \pi_s(w)$, and $\mathsf{v}$ contains a reduced word for $y$, so  $D\in\Delta(\Q',\tau_s(F))$. On the other hand, if $\mathsf{w}$ does not use the initial $s$, then $s\Q'\vert_{V'\setminus D}$ already contains $\mathsf{w}$, which in turn contains a reduced word for $\pi_s(w)$ since $\pi_s(w)\leqb w$. In this case, we again conclude that $D\in\Delta(\Q',\tau_s(F))$. This shows that $\del_{k_1}(\Delta(\Q,F))\subseteq\Delta(\Q',\tau_s(F))$. 

To prove the reverse containment, assume now that $D\in\Delta(\Q',\tau_s(F))$. Then $\Q'\vert_{V'\setminus D}$ contains a reduced word $\mathsf{u}$ for an element $u\in\min(\tau_s(F))$. We have $u=\pi_s(u')$ for some $u'\in\min(F)$. Since the word $s\mathsf{u}$ contains a reduced word for $u'$, it follows that $s\Q'\vert_{V'\setminus D}$ contains a reduced word for $u'$. Hence, $D\in\del_{k_1}(\Delta(\Q,F))$. 
\end{proof} 

\begin{proof}[Proof of \cref{thm:vertex}]
Fix $F\in\Mac(W)$ and a finite word $\Q$ over $S$ with index set $V\subset \mathbb Z$. Assume that for every $F'\in\Mac(W)$ with $F'\leqr F$, all elements of $\min(F')$ have the same Coxeter length. Our goal is to prove that $\Delta(\Q,F)$ is vertex-decomposable. We proceed by induction on $|V|$. 

Write $V=\{k_1<\cdots<k_t\}$, and let $V'=V\setminus\{k_1\}$. Write $\Q=s\Q'$, where $s\in S$ and $\Q'=\Q\vert_{V'}$. It is immediate from the definitions that \[\lk_{k_1}(\Delta(\Q,F))=\Delta(\Q',F),\] so we know by induction that $\lk_{k_1}(\Delta(\Q,F))$ is vertex-decomposable. \cref{lem:deletion} tells us that \[\del_{k_1}(\Delta(\Q,F))=\Delta(\Q',\tau_s(F)).\] For every $F'\in\Mac(W)$ with $F'\leqr \tau_s(F)$, we have $F'\leqr F$, so all elements of $\min(F')$ have the same Coxeter length. Thus, we can apply induction to deduce that $\del_{k_1}(\Delta(\Q,F))$ is vertex-decomposable. 

If $s\not\in\Des(F)$, then $\tau_s(F)=F$, so $\del_{k_1}(\Delta(\Q,F))=\Delta(\Q',F)=\lk_{k_1}(\Delta(\Q,F))$. This says that every face in $\Delta(\Q',F)$ can be extended to a face of $\Delta(\Q,F)$ by adding $k_1$, so $\Delta(\Q,F)$ is a cone over $\Delta(\Q',F)$. It is a standard fact that a cone over a vertex-decomposable simplicial complex is vertex-decomposable; hence, $\Delta(\Q,F)$ is vertex-decomposable in this case. 

Now assume $s\in\Des(F)$ so that $\tau_s(F)\neq F$. Let $D\subseteq V'$ be a facet of $\Delta(\Q',\tau_s(F))$; we aim to show that $D$ is also a facet of $\Delta(\Q,F)$. Note that $\Q'\vert_{V'\setminus D}$ is a reduced word $\mathsf{w}$ for an element $w\in\min(\tau_s(F))$. All minimal elements of $\tau_s(F)$ have the same Coxeter length, say $m$. Similarly, all minimal elements of $F$ have the same Coxeter length. Since $\min(\tau_s(F))\subseteq\{\pi_s(u):u\in\min(F)\}$ and $\tau_s(F)\neq F$, the minimal elements of $F$ all have Coxeter length $m+1$. Consequently, $w=\pi_s(sw)$, where $sw\in\min(F)$. It follows that $\Q\vert_{V\setminus D}=s\mathsf{w}$ is a reduced word for $sw$, so $D$ is a facet of $\Delta(\Q,F)$. 
\end{proof}

The preceding inductive proof of \cref{thm:vertex} is somewhat routine. A similar argument appears in Jahn's thesis \cite{Jahn}, which is a precursor for studying unions of subword complexes.

\section{Cohen--Macaulay ASM Varieties}\label{sec:CM} 

In this section, we will prove \cref{thm:CM}. We first need a few more definitions. 

Let $\mathcal X_V=\{x_v:v\in V\}$ be a set of variables indexed by a finite set $V$, and let $\mathbb K[\mathcal X_V]$ be the corresponding polynomial ring over the field $\mathbb K$. For $D\subseteq V$, let $x_D=\prod_{v\in D}x_v$. Given a simplicial complex $\Delta$ with vertex set $V$, we define the \dfn{Stanley--Reisner ideal} to be the ideal 
\[I_\Delta:=(x_D:D\not\in \Delta)\subseteq\mathbb K[\mathcal X_V]\]  generated by the monomials indexed by the non-faces of $\Delta$. The \dfn{Stanley--Reisner ring} of $\Delta$ is 
\[\mathbb K[\Delta]:=\mathbb K[\mathcal X_V]/I_\Delta.\] We say $\Delta$ is \dfn{pure} if all of its facets have the same cardinality.  

A Noetherian ring is \dfn{Cohen--Macaulay} if each of its localizations at a prime ideal has depth equal to its Krull dimension.  A simplicial complex $\Delta$ is \dfn{Cohen--Macaulay} over $\mathbb K$ if its
Stanley--Reisner ring $\mathbb K[\Delta]$ is Cohen--Macaulay.  We will use the standard fact that a pure vertex-decomposable simplicial complex is Cohen--Macaulay over every field.

If $\preceq$ is a term order on a polynomial ring and $I$ is an ideal, then
$\mathrm{in}_{\preceq}(I)$ denotes the initial ideal of $I$, which is the ideal generated by the
initial monomials of the elements of $I$.  An \dfn{antidiagonal term order} on the polynomial ring 
$\RR_n:=\mathbb K[z_{i,j}:i,j\in[n]]$ is a term order with the property that the initial monomial of each
minor of the matrix $Z=(z_{i,j})_{i,j\in[n]}$ is the product of the entries on its antidiagonal.  Explicitly, if the minor uses
rows $a_1<\cdots<a_r$ and columns $b_1<\cdots<b_r$, then its antidiagonal monomial is
\[
z_{a_1,b_r}z_{a_2,b_{r-1}}\cdots z_{a_r,b_1}.
\]

\begin{proof}[Proof of \cref{thm:CM}]
Let $A\in\ASM_n$. Escobar, Klein, and Weigandt already proved the equivalence of the statements \eqref{(i)}, \eqref{(ii')}, and \eqref{(ii)} (see \cite[Propositions~2.4~\&~3.20]{EKW}). They also noted that \eqref{(iii)} implies \eqref{(ii)} because Cohen--Macaulay varieties with standard graded coordinate rings are always equidimensional. Thus, we just need to show that \eqref{(ii')} implies \eqref{(iii)}. 

Suppose that for every $B\in[e,A]_{\mathsf{L}}$, all minimal elements of the order filter $\Psi(B)$ have the same Coxeter length. Fix $B\in[e,A]_{\mathsf{L}}$, and let $F=\Psi(B)$. It follows from \cref{thm:vertex} that $\Delta(\Q,F)$ is vertex-decomposable for every finite word $\Q$ over the simple reflections of $S_n$. We will prove that $X_B$ is Cohen--Macaulay.

Let 
\[
\Lambda_n=\{(i,j)\in[n]\times[n]: i+j\le n\} 
\]
be the staircase Young diagram (in English conventions), where $(i,j)$ represents the box in row $i$ and column $j$. Order $\Lambda_n$ by reading rows from top to bottom and, within each row, from right to left.
Mapping each pair $(i,j)$ to the simple reflection $s_{i+j-1}$ yields the word
\[
\Q_n=(s_{n-1}s_{n-2}\cdots s_1)(s_{n-1}s_{n-2}\cdots s_2)\cdots(s_{n-1}). \] 
This word $\Q_n$ is a reduced word for the long element of $S_n$, so every permutation in $S_n$ has a reduced word appearing as a subword of $\Q_n$.  By assumption, all elements of $\min(F)$ have the same Coxeter length, say $m$.  Every facet of $\Delta(\Q_n,F)$ is the complement
of a reduced word for some element of $\min(F)$, so every facet has cardinality $\binom{n}{2}-m$. Thus, $\Delta(\Q_n,F)$ is pure.  We view the Stanley--Reisner ideal $I_{\Delta(\Q_n,F)}$ as an ideal of $\RR_n':=\mathbb K[z_{i,j}:(i,j)\in\Lambda_n]$, which is contained in the polynomial ring $\RR_n=\mathbb K[z_{i,j}:i,j\in[n]]$. Since $\Delta(\Q_n,F)$ is vertex-decomposable, its
Stanley--Reisner ring $\RR_n'/I_{\Delta(\Q_n,F)}=\mathbb K[\Delta(\Q_n,F)]$ is Cohen--Macaulay. 

We now relate the Stanley--Reisner ring $\mathbb K[\Delta(\Q_n,F)]$ to the ASM ideal $I_B$.  Weigandt \cite[Proposition~5.4]{Wei17} proved that the minimal prime
decomposition of $I_B$ is given by 
\[
I_B=\bigcap_{w\in \min(F)} I_w.
\] Let $\preceq$ be an antidiagonal term order on $\RR_n$. 
Axelrod--Freed, Hao, Kendall, Klein, and Luo \cite[Proposition~3.8]{AFHKKL} proved that
\[
\operatorname{in}_\preceq(I_B)
=
\bigcap_{w\in \min(F)} \operatorname{in}_\preceq(I_w).
\]
For each $w\in S_n$, Knutson and Miller's pipe-dream degeneration theorem \cite[Theorem~B]{KM05} identifies $\operatorname{in}_\preceq(I_w)$ with
the Stanley--Reisner ideal of the subword complex $\Delta(\Q_n,w)$.  Hence,
\[
\operatorname{in}_\preceq(I_B)
=
\bigcap_{w\in \min(F)}I_{\Delta(\Q_n,w)}\RR_n
=
I_{\Delta(\Q_n,F)}\RR_n. 
\]
It follows that $\RR_n/\operatorname{in}_{\preceq}(I_B)$ is a polynomial extension of
$\mathbb K[\Delta(\Q_n,F)]$.  Therefore $\RR_n/\operatorname{in}_{\preceq}(I_B)$ is Cohen--Macaulay.  Axelrod--Freed, Hao, Kendall, Klein, and Luo \cite[Proposition~3.10]{AFHKKL} proved that $\RR_n/I_B$ is Cohen--Macaulay if and only if $\RR_n/\operatorname{in}_\preceq (I_B)$ is Cohen--Macaulay.  This shows that $\RR_n/I_B$ is Cohen--Macaulay, so
$X_B$ is Cohen--Macaulay. 
\end{proof}

\section{MacNeille Pop-Stack Operators}\label{sec:pop} 
Let us now assume that the Coxeter system $(W,S)$ is finite and irreducible. For $J\subseteq S$, let $W_J$ be the parabolic subgroup generated by $J$, and let $w_\circ(J)$ be the long element of $W_J$. Recall from \cref{def:pop} that we define the \dfn{MacNeille pop-stack operator} $\MacPop\colon \Mac(W)\to\Mac(W)$ by 
\[\MacPop(F)=\tau_{w_\circ(\Des(F))}(F).\] 
The \dfn{pop-stack operator} of $W$ is the map $\Pop\colon W\to W$ defined by \[\Pop(x)=\pi_{w_\circ(\Des(x))}(x).\] For each $u\in W$, the fact that $\pi_u$ is order-preserving implies that $\iota(\pi_u(x))=\tau_u(\iota(x))$ for every $x\in W$. In particular,  
\[
\iota(\Pop(x))=\MacPop(\iota(x))
\] for every $x\in W$.

Let $W^J$ denote the set of minimum-length representatives of the left cosets in $W/W_J$. Every element $w\in W$ has a unique factorization of the form $w=w^Jw_J$, where $w^J\in W^J$ and $w_J\in W_J$. 

\begin{lemma}\label{lem:commuting-descents}
Fix $s\in S$, and let $J=S\setminus\{s\}$. For every integer $t\geq 0$, the elements of $\Des(\Pop^t(w_\circ^J))$ commute pairwise.
\end{lemma}

\begin{proof}
First suppose that the Coxeter number $h$ of $W$ is even. The desired statement then follows from \cite[Proposition~3.4]{DefantCoxeterPop} after applying inversion to switch from right descents and right weak order to left descents and left weak order.

Now suppose that $h$ is odd. By the classification of finite irreducible Coxeter systems, $W$ is either of type $A_{h-1}$ or of type $I_2(h)$. First assume $W$ is of type $A_{h-1}$ so that we may identify $W$ with $S_h$ and write $S=\{s_1,\ldots,s_{h-1}\}$, where $s_i=(i\,\,i+1)$. Say $s=s_i$. The set $W^J$ is an order ideal in left weak order, and $\Pop(v)\leqr v$ for every $v\in W$. Hence, if
$v=\Pop^t(w_\circ^J)$, then $v\in W^J$. Consequently,
we have $v(1)<\cdots<v(i)$ and 
$v(i+1)<\cdots<v(h)$. If $s_m$ and $s_{m+1}$ were both left descents of $v$, then we would have
\[
v^{-1}(m)>v^{-1}(m+1)>v^{-1}(m+2).
\]
Among the three entries $m,m+1,m+2$, two would lie in the same one of the two increasing blocks displayed above, which is impossible. Thus, no two elements of $\Des(v)$ are adjacent in the Coxeter diagram, so they commute pairwise.

Finally, suppose that $W$ is of type $I_2(h)$. The element $w_\circ$ is the only element of $W$ with both simple reflections as left descents. Since $w_\circ^J\neq w_\circ$ and $\Pop^t(w_\circ^J)\leqr w_\circ^J$, the descent set of $\Pop^t(w_\circ^J)$ has at most one element. Hence, its elements commute pairwise.
\end{proof}

\begin{lemma}\label{lem:induction} 
Let $F\in\Mac(W)$. Fix $s\in S$, and let $J=S\setminus\{s\}$. For each integer $t\geq 0$, there exists a minimal element $b$ of the order filter $\MacPop^t(F)$ such that $b^J\leqb\Pop^t(w_\circ^J)$. 
\end{lemma}

\begin{proof}
We proceed by induction on $t$. If $t=0$, then the result is immediate because $F$ is nonempty and $w_\circ^J$ is the maximum
element of $W^J$.

Now assume $t\geq 1$, and let
\[
F_{t-1}=\MacPop^{t-1}(F)
\quad\text{and}\quad
K=\Des(F_{t-1}).
\]
By the induction hypothesis, there exists a minimal element $a$ of $F_{t-1}$ such that
\[
a^J\leqb \Pop^{t-1}(w_\circ^J).
\]
By \cref{lem:commuting-descents}, all of the left descents of $\Pop^{t-1}(w_\circ^J)$ commute with each other. Therefore, the left-handed version of \cite[Lemma~3.5]{DefantCoxeterPop}, obtained by applying inversion, yields that
\[
\Pop(a^J)\leqb \Pop^t(w_\circ^J).
\]

Now let $x=\pi_{w_\circ(K)}(a)$. Since $a\in F_{t-1}$, we have
\[
x\in \pi_{w_\circ(K)}(F_{t-1})\subseteq \nabla(\pi_{w_\circ(K)}(F_{t-1}))=\MacPop(F_{t-1})=\MacPop^t(F).
\]
Choose a minimal element $b$ of the order filter $\MacPop^t(F)$ such that $b\leqb x$.

We claim that
$b^J\leqb \Pop(a^J)$.
To prove this, we first show that $\Des(a)\subseteq K$. Indeed, let $r\in \Des(a)$. Then $\pi_r(a)=ra\leqb a$. Because $a$ is a minimal element of the order filter $F_{t-1}$, it follows that $ra\notin F_{t-1}$. On the
other hand,
\[
ra=\pi_r(a)\in \pi_r(F_{t-1})\subseteq \nabla(\pi_r(F_{t-1}))=\tau_r(F_{t-1}).
\]
Thus, $\tau_r(F_{t-1})\neq F_{t-1}$, so $r\in \Des(F_{t-1})=K$, as desired.

Since $\Des(a)\subseteq K$, we have
\[
w_\circ(\Des(a))\leqr w_\circ(K).
\]
Consequently,
\[
x=\pi_{w_\circ(K)}(a)\leqr \pi_{w_\circ(\Des(a))}(a)=\Pop(a).
\]
Applying the parabolic projection $w\mapsto w^J$, which is order-preserving for left weak order,
we find that 
$
x^J\leqr (\Pop(a))^J$.
It follows from \cite[Lemma~3.6]{DefantCoxeterPop} that
$(\Pop(a))^J\leqr \Pop(a^J)$.
Hence,
\[
x^J\leqr \Pop(a^J).
\]
Since weak order refines Bruhat order, it follows that
\[
x^J\leqb \Pop(a^J).
\] Finally, because $b\leqb x$ and the map $w\mapsto w^J$ is order-preserving for Bruhat order,
we have
\[
b^J\leqb x^J\leqb \Pop(a^J)\leqb \Pop^t(w_\circ^J),
\] as desired. 
\end{proof}

We are now in a position to prove \cref{thm:pop}, which states that $\G_{\Mac(W)}=h-1$, where $\G_{\Mac(W)}$ is the maximum number of iterations of $\MacPop$ needed to send elements of $\Mac(W)$ to the bottom element $W$ and $h$ is the Coxeter number of $W$.

\begin{proof}[Proof of \cref{thm:pop}]
We already saw in \eqref{eq:h} that $h-1\leq\G_{\Mac(W)}$, so we just need to prove the reverse inequality. Fix $F\in\Mac(W)$; our goal is to show that $\MacPop^{h-1}(F)=W$. 

For each $s\in S$, we know by \cref{lem:induction} that there exists an element $b_s\in\MacPop^{h-1}(F)$ such that $b_s^{S\setminus\{s\}}\leqb\Pop^{h-1}(w_\circ^{S\setminus\{s\}})$. We also know that $\Pop^{h-1}(w_\circ^{S\setminus\{s\}})=e$ since it is known that $\G_W=h-1$ \cite[Theorem~1.3]{DefantCoxeterPop}. It follows that $b_s\in W_{S\setminus\{s\}}$ for every $s\in S$. Thus, for every $s\in S$, the order filter $\MacPop^{h-1}(F)$ contains an element of $W_{S\setminus\{s\}}$. Each parabolic subgroup $W_{S\setminus\{s\}}$ is an order ideal in the Bruhat order, so every element of $\L(\MacPop^{h-1}(F))$ belongs to the intersection $\bigcap_{s\in S}W_{S\setminus\{s\}}$. However, this intersection is $\{e\}$, so $\L(\MacPop^{h-1}(F))=\{e\}$. Finally, 
\[\MacPop^{h-1}(F)=\U(\L(\MacPop^{h-1}(F)))=\U(\{e\})=W,\] as desired. 
\end{proof}

\section{MacNeille Complexes}\label{sec:complex}
Throughout this section, assume that $(W,S)$ is finite. For $J\subseteq S$, let $W_J$ be the parabolic subgroup generated by $J$, let $w_\circ(J)$ be the long element of $W_J$, and define $p_J=\tau_{w_\circ(J)}\colon\Mac(W)\to\Mac(W)$.
The $0$-Hecke relations imply that
\begin{equation}\label{eq:parabolic-absorption}
p_K\circ p_J=p_K\qquad\text{whenever }J\subseteq K\subseteq S.
\end{equation}
In particular, each $p_J$ is idempotent.

Let
\[
\mathcal F_{\Mac}(W,S)=\{(J,E):J\subseteq S\text{ and }E\in p_J(\Mac(W))\}.
\]
We partially order $\mathcal F_{\Mac}(W,S)$ by declaring that $(J',E')\leq (J,E)$ if and only if $J'\supseteq J$ and $E'=p_{J'}(E)$.
It follows from \eqref{eq:parabolic-absorption} that this is indeed a partial order. Because $\pi_{w_\circ}(w)=e$ for every $w\in W$, we have $p_S(F)=W$ for every $F\in\Mac(W)$. Hence, the unique minimum element is $(S,W)$, while the maximal elements are the pairs $(\varnothing,F)$ with $F\in\Mac(W)$. Moreover, the interval below $(J,E)$ is $\{(K,p_K(E)):J\subseteq K\subseteq S\}$, which is isomorphic to the Boolean lattice $2^{S\setminus J}$. Thus, $\mathcal F_{\Mac}(W,S)$ is a simplicial poset. We call its associated Boolean cell complex the \dfn{MacNeille complex} of $(W,S)$ and denote it by $\Delta_{\Mac}(W,S)$.

For $s\in S$ and a facet $F\in\Mac(W)$, let $v_s(F)=(S\setminus\{s\},p_{S\setminus\{s\}}(F))$. Let
\[C_F=\{v_s(F):s\in S\}.\] Then $C_F$ is the vertex set of $F$. By \eqref{eq:iota}, we have
\[p_J(\iota(w))=\iota(\pi_{w_\circ(J)}(w))\]
for all $J\subseteq S$ and $w\in W$. Consequently, the facets indexed by the principal filters $\iota(w)$ form a subcomplex naturally isomorphic to the Coxeter complex of $(W,S)$.

When $W=S_n$, the complex $\Delta_{\Mac}(W,S)$ recovers the complex studied by Hamaker and Reiner. Indeed, under the identification of $\Mac(S_n)$ with monotone triangles mentioned in \cref{subsec:1.3}, \cite[Proposition~4.1]{HR} states that $p_J$ fixes the rows corresponding to the simple reflections outside of $J$ and replaces the other rows by their unique componentwise-minimal completion. Hence, the face $(J,p_J(F))$ records the partial monotone triangle obtained by deleting the rows corresponding to the elements of $J$.

We now prove the identity that will allow us to recover each face from its vertices. For $I\subseteq S$, let
\[\rho_I=\pi_{w_\circ(I)}\colon W\to{}^I W,\]
where ${}^I W$ is the set of minimum-length representatives of the cosets in $W_I\backslash W$. Thus, $\rho_I(w)$ is the minimum-length element of the coset $W_Iw$.

\begin{lemma}\label{lem:projections}
For every $F\in\Mac(W)$ and $J\subseteq S$, we have
\[p_J(F)=\bigcap_{s\in S\setminus J}p_{S\setminus\{s\}}(F),\] 
where the intersection on the right is understood to be $W$ when $J=S$. 
\end{lemma}

\begin{proof}
Iterating \cref{Lem2} along a reduced word for $w_\circ(I)$, we find that $p_I(F)=\U(\rho_I(\L(F)))$ for every $I\subseteq S$. Consequently,
\begin{equation}\label{eq:parabolic-membership}
y\in p_I(F)\text{ if and only if }\rho_I(d)\leqb y\text{ for every }d\in\L(F).
\end{equation}

Suppose first that $y\in p_J(F)$, and fix $s\in S\setminus J$. For each $d\in\L(F)$, we have $\rho_J(d)\leqb y$. Because $J\subseteq S\setminus\{s\}$, we also have $\rho_{S\setminus\{s\}}\circ\rho_J=\rho_{S\setminus\{s\}}$. It follows that
\[
\rho_{S\setminus\{s\}}(d)=\rho_{S\setminus\{s\}}(\rho_J(d))\leqb\rho_J(d)\leqb y.
\]
Hence, $y\in p_{S\setminus\{s\}}(F)$ by \eqref{eq:parabolic-membership}. This proves one containment.

Conversely, suppose that $y\in p_{S\setminus\{s\}}(F)$ for every $s\in S\setminus J$. Fix $d\in\L(F)$. By \eqref{eq:parabolic-membership}, we have 
$\rho_{S\setminus\{s\}}(d)\leqb y$
for every $s\in S\setminus J$. Applying the order-preserving map $\rho_{S\setminus\{s\}}$ to each of these inequalities and using idempotence, we obtain that
$\rho_{S\setminus\{s\}}(d)\leqb\rho_{S\setminus\{s\}}(y)$ for every $s\in S\setminus J$.
Using $\rho_{S\setminus\{s\}}\circ\rho_J=\rho_{S\setminus\{s\}}$, we can rewrite these inequalities as
\[
\rho_{S\setminus\{s\}}(\rho_J(d))\leqb\rho_{S\setminus\{s\}}(\rho_J(y))\] for every $s\in S\setminus J$. Both $\rho_J(d)$ and $\rho_J(y)$ belong to ${}^J W$. It follows from \cite[Corollary~1.4]{Stembridge} that
\[\rho_J(d)\leqb\rho_J(y)\leqb y.\]
Since this holds for every $d\in\L(F)$, \eqref{eq:parabolic-membership} implies that $y\in p_J(F)$.
\end{proof}

We are now in a position to prove \cref{thm:complex}.

\begin{proof}[Proof of \cref{thm:complex}]
Consider a face $(J,E)$ of $\Delta_{\Mac}(W,S)$. Its vertices are
\[
\{(S\setminus\{s\},p_{S\setminus\{s\}}(E)):s\in S\setminus J\}.
\]
These vertices determine $J$. Because $E\in\operatorname{im}(p_J)$, we have $p_J(E)=E$, so \cref{lem:projections} gives
\[
E=\bigcap_{s\in S\setminus J}p_{S\setminus\{s\}}(E).
\]
Thus, every face is determined by its set of vertices. It follows that the Boolean cell complex $\Delta_{\Mac}(W,S)$ is an abstract simplicial complex. Each facet has one vertex of type $S\setminus\{s\}$ for every $s\in S$, so this complex is pure of dimension $|S|-1$.

It remains to prove shellability. We follow the argument in \cite[Section~4]{HR}. Let $F_1,\ldots,F_N$ be a linear extension of the weak order on $\Mac(W)$, and write $C_i=C_{F_i}$. A standard criterion for shellings states that it suffices to prove that whenever $i<j$, there exists $k<j$ such that $C_i\cap C_j\subseteq C_k\cap C_j$ and $C_k\cap C_j$ is a ridge of $C_j$.

Fix $i<j$, and let $J=\{s\in S:v_s(F_i)\neq v_s(F_j)\}$.
We claim that $J\cap\Des(F_j)$ is nonempty. Suppose otherwise. Then $\tau_s(F_j)=F_j$ for every $s\in J$, so $p_J(F_j)=F_j$. On the other hand, \cref{lem:projections} and the definition of $J$ imply that
\[p_J(F_i)
=\bigcap_{s\in S\setminus J}p_{S\setminus\{s\}}(F_i)=\bigcap_{s\in S\setminus J}p_{S\setminus\{s\}}(F_j)
=p_J(F_j)=F_j.
\] Hence, $F_j\leqr F_i$, contradicting the fact that $F_i$ appears before $F_j$ in a linear extension.

Choose $s\in J\cap\Des(F_j)$, and let $F_k=\tau_s(F_j)$. Then $F_k<_{\mathsf L}F_j$, so $k<j$. If $t\neq s$, then $s\in S\setminus\{t\}$, and \eqref{eq:parabolic-absorption} gives
\[
p_{S\setminus\{t\}}(F_k)=p_{S\setminus\{t\}}(\tau_s(F_j))=p_{S\setminus\{t\}}(F_j).
\]
Thus, $C_k$ and $C_j$ have the same vertex of type $S\setminus\{t\}$ for every $t\neq s$. Since $F_k\neq F_j$ and facets are determined by their vertex sets, their vertices of type $S\setminus\{s\}$ are different. Therefore,
\[
C_k\cap C_j=C_j\setminus\{v_s(F_j)\},
\]
which is a ridge of $C_j$. Since $s\in J$, we also have $C_i\cap C_j\subseteq C_k\cap C_j$. This proves that $C_1,\ldots,C_N$ is a shelling.
\end{proof}

\section{AI Usage}\label{sec:AI}  
The initial seed of this project was planted in 2022, when Zach Hamaker told the author about an idea he and Vic Reiner had discussed that did not make it into their paper. He explained how one could extend the ASM weak order to the setting of arbitrary Coxeter groups, provided one could prove \cref{thm:descents}. The author found a proof, which he kept as a note on his computer. In 2023, an undergraduate at MIT asked to do research with the author, who suggested trying to prove \cref{thm:pop} in type~$A$ using the combinatorics of alternating sign matrices and monotone triangles. It turns out this was the wrong approach; the author later realized how to give a type-uniform proof using a subtle inductive argument that passes back and forth between the weak order and the Bruhat order. By that time, the student had pivoted to a different project, so the author added a sketch of the proof of \cref{thm:pop} to his earlier notes. 

In May of 2026, the author uploaded the \LaTeX\ file of his notes to ChatGPT~5.4 Pro and asked it to fill in the gaps. Not only did it add in the details perfectly, but it also caught numerous typos. It rewrote the notes in a matter of minutes. At this point, the author asked ChatGPT to suggest other directions to explore related to the weak order on MacNeille completions. ChatGPT suggested proving the conjecture of Hamaker and Reiner regarding the poset topology of open intervals in the ASM weak order. The author asked ChatGPT to prove this conjecture. Instead, it returned the counterexample shown in \cref{fig:counterexample}. 

At this point, the author opened a new conversation with ChatGPT and uploaded his newly polished notes and the article of Escobar, Klein, and Weigandt \cite{EKW}. He asked, ``Could some of the ideas from the paper of Escobar, Klein, and Weigandt generalize to other Coxeter groups using my notion of weak order on $\Mac(W)$? Could we define analogues of ASM varieties in other types?'' ChatGPT responded by suggesting definitions of ASM varieties for other Weyl groups using unions of opposite Schubert varieties in the flag variety $G/B$. The author then asked, ``Can you find a counterexample to Conjecture~3.21 in the EKW paper?'' After thinking for around 50 minutes, ChatGPT responded by saying it had not found a counterexample and that it verified the conjecture for all $n\leq 7$. The author asked, ``In that case, can you prove Conjecture~3.21?'' After thinking for around 19 minutes, ChatGPT responded with the main ingredients of the proof. The proof of \cref{thm:CM} presented in \cref{sec:CM} is essentially the one that ChatGPT found, although the author massaged the writing style and added several details.
In the end, this proof is not terribly intricate, and several of the steps are somewhat routine. Probably the most significant advance is \cref{lem:deletion}.

After the first version of this paper appeared, Vic Reiner sent the author an earlier draft of \cite{HR} containing the MacNeille complex construction and asked when the resulting Boolean cell complex is simplicial. The author discussed this question with ChatGPT. Its initial attempts at a proof incorrectly treated $\nabla(X)$ as the set of common upper bounds of $X$, rather than the order filter generated by $X$; the author pointed out this error. The corrected discussion led to the maximal-parabolic identity in \cref{lem:projections}, which proves simpliciality. Once this identity was in place, the shelling argument from \cite[Section~4]{HR} carried over essentially verbatim, yielding \cref{thm:complex}.

\section*{Acknowledgments} 
The author thanks Zach Hamaker and Vic Reiner for suggesting the Coxeter-theoretic definition of weak order on $\Mac(W)$, and he thanks Vic Reiner for bringing the MacNeille complex construction to his attention. He thanks Katherine Tung for pointing him to Jahn's thesis. He thanks an anonymous referee for helpful comments. He also thanks OpenAI for providing unlimited access to ChatGPT~5.4 Pro.


\begin{thebibliography}{999999999} 
\bibitem[AFHKKL]{AFHKKL}
I. Axelrod-Freed, H. Hao, M. Kendall, P. Klein, and Y. Luo. Some algebraic properties of ASM varieties. \emph{Glasg.\ Math.\ J.}, {\bf 68} (2026), 440--459.  

\bibitem[BB05]{BB}
A. Bj\"orner and F. Brenti. Combinatorics of Coxeter groups, vol. 231 of Graduate Texts in Mathematics. Springer, 2005. 

\bibitem[BC17]{BC17} 
N. Bergeron and C. Ceballos. A Hopf algebra of subword complexes. \emph{Adv. Math.}, {\bf 305} (2017), 1163--1201. 

\bibitem[CG19]{CG}
A. Claesson and B.\'A. Gu{\dh}mundsson. Enumerating permutations sortable by $k$ passes through a pop-stack. \emph{Adv. Appl. Math.}, {\bf 108} (2019), 79--96. 

\bibitem[CGP23]{CGP} 
A. Claesson, B. \'A. Gu{\dh}mundsson, and J. Pantone. Counting pop-stacked permutations in polynomial time. \emph{Exp. Math.}, {\bf 32} (2023), 97--104. 

\bibitem[Def22A]{DefantMeeting}
C. Defant. Meeting covered elements in $\nu$-Tamari lattices.\ \emph{Adv.\ Appl.\ Math.}, {\bf 134} (2022).  

\bibitem[Def22B]{DefantCoxeterPop}
C. Defant. Pop-stack-sorting for Coxeter groups. \emph{Comb. Theory}, {\bf 2} (2022). 

\bibitem[EG21]{Elder}
M. Elder and Y. K. Goh. $k$-pop stack sortable permutations and 2-avoidance. \emph{Electron. J. Combin.}, {\bf 28} (2021). 

\bibitem[EKW]{EKW}
L. Escobar, P. Klein, and A. Weigandt. Algebra and geometry of ASM weak order. arXiv:2502.19266. 

\bibitem[EM18]{EM}
L. Escobar and K. M\'esz\'aros. Subword complexes via triangulations of root polytopes. \emph{Algebr. Comb.}, {\bf 1} (2018), 395--414. 

\bibitem[Ful92]{Fulton}
W. Fulton. Flags, Schubert polynomials, degeneracy loci, and determinantal formulas. \emph{Duke Math. J.}, {\bf 65} (1992), 381--420.  

\bibitem[HR20]{HR} 
Z. Hamaker and V. Reiner. Weak order and descents for monotone triangles. \emph{European J. Combin.}, {\bf 86} (2020). 

\bibitem[Jah22]{Jahn}
D. Jahn. Combinatorics of subword complexes for finite type Coxeter systems. Ph.D. thesis, Ruhr-Universit\"at Bochum, 2022.

\bibitem[KM04]{KM04}
A. Knutson and E. Miller. Subword complexes in Coxeter groups. \emph{Adv. Math.}, {\bf 184} (2004), 161--176. 

\bibitem[KM05]{KM05}
A. Knutson and E. Miller. Gr\"obner geometry of Schubert polynomials. \emph{Ann. of Math.}, {\bf 161} (2005), 1245--1318. 

\bibitem[PS15]{PS15} 
V. Pilaud and C. Stump. Brick polytopes of spherical subword complexes and generalized associahedra. \emph{Adv. Math.}, {\bf 276} (2015), 1--61. 

\bibitem[STW25]{STW25} 
C. Stump, H. Thomas, and N. Williams. Cataland: why the {F}u{\ss}? \emph{Mem. Amer. Math. Soc.}, {\bf 305} (2025). 

\bibitem[Ste05]{Stembridge}
J. R. Stembridge. Tight quotients and double quotients in the Bruhat order. \emph{Electron. J. Combin.}, {\bf 11} (2004/06), no.~2, Research Paper~14. 

\bibitem[Wei17]{Wei17}
A. Weigandt. Prism tableaux for alternating sign matrix varieties. arXiv:1708.07236. 
\end{thebibliography}
\end{document}